\documentclass{amsart}
\usepackage{graphicx,amssymb,epsfig,amsmath}
\usepackage{lipsum}
\usepackage{enumitem}
\usepackage{hyperref}
\usepackage{xcolor}
\usepackage{amsmath}
\usepackage{amssymb,amsmath}
\usepackage[all]{xy}

\usepackage{xkeyval,tkz-base}
\usetikzlibrary{arrows,calc}
\usepackage{tikz,pgfplots}
\pgfplotsset{compat=1.10}
\usepgfplotslibrary{fillbetween}
\usetikzlibrary{patterns}
\usepgflibrary{fpu}
 \usetikzlibrary{calc,shapes,decorations,decorations.text, mindmap,shadings,patterns,matrix,arrows,intersections,automata,backgrounds}
\usepackage{float}
\usetikzlibrary{arrows}
\usetikzlibrary{decorations.pathmorphing}
 \usetikzlibrary{plotmarks}
        \usetikzlibrary{intersections}
        \usetikzlibrary{shadings}
        \usetikzlibrary{calc}
        \usetikzlibrary{patterns}

\newtheorem{theorem}{Theorem}[section]
\newtheorem{proposition}[theorem]{Proposition}
\newtheorem{lemma}[theorem]{Lemma}
\newtheorem{corollary}[theorem]{Corollary}
\newtheorem{remark}[theorem]{Remark}
\newtheorem{example}[theorem]{Example}

\newtheorem{definition}[theorem]{Definition}

\newtheorem{corollaries}[theorem]{Corollaries}
\newtheorem{theoremf}{Folk Formula}[section]
\newtheorem{maintheorem}{Main Theorem}[section]
\newcommand{\bth}{\begin{theorem}}
\newcommand{\bpr}{\begin{proposition}}
\newcommand{\epr}{\end{proposition}}
\newcommand{\bco}{\begin{corollary}}
\newcommand{\eco}{\end{corollary}}
\newcommand{\ble}{\begin{lemma}}
\newcommand{\ele}{\end{lemma}}
\newcommand{\bde}{\begin{definition}\rm}
\newcommand{\ede}{\end{definition}\rm}
\newcommand{\bre}{\begin{remark}\rm}
\newcommand{\ere}{\end{remark}}
\newcommand{\bex}{\begin{example}\rm}
\newcommand{\eex}{\end{example}}
\newcommand{\bcors}{\begin{corollaries}\rm}
\newcommand{\ecors}{\end{corollaries}}
\newcommand{\bthf}{\begin{theoremf}}
\newcommand{\bmain}{\begin{maintheorem}}
\newcommand{\emain}{\end{maintheorem}}

\def\la#1{\hbox to #1pc{\leftarrowfill}}
\def\ra#1{\hbox to #1pc{\rightarrowfill}}

\def\lrar{{\ra 2}}

\def\sp#1{\hbox{SP}^{#1}}

\def\cp#1{\hbox{CP}^{#1}}
\def\rat#1{\hbox{Rat}_{#1}}

\def\sing#1{\hbox{Sing}_{#1}}
\def\bbz{{\mathbb Z}}

\def\bbr{{\mathbb R}}
\def\bbc{{\mathbb C}}
\def\bbp{{\mathbb P}}

\def\conf{\hbox{Conf}}
\def\bbl{{\mathcal L}}
\def\bbm{{\mathcal M}}

\def\bbq{{\mathbb Q}}

\begin{document}

\title{Combinatorial Invariants of Stratifiable Spaces II}

\author{Sadok Kallel and
  Faten Labassi}
  \address{First author: American University of Sharjah, UAE, and Laboratoire Painlev\'e, Universit\'e de Lille, France}
  \email{sadok.kallel@univ-lille.fr}
  \address{Second author:  Imam Mohammad Ibn  Saud Islamic University (IMSIU), Saudi Arabia, and Department of Mathematics, Faculty of Sciences of Tunis, University of Tunis El Manar, Tunisia}
  \email{fmlabassi@imamu.edu.sa}

\maketitle

\begin{abstract}
In this follow-up to \cite{kt}, we continue developing the notion of a lego category and its many applications  to stratifiable spaces and the computation of their Grothendieck classes. We illustrate the effectiveness of this construction by giving very short derivations of the class of a quotient by the ``stratified action'' of a discrete group \cite{akita}, the class of a crystallographic quotient, the class of both a polyhedral product and a polyhedral (or simplicial) configuration space \cite{cooper}, the class of a permutation product \cite{macdo} and, foremost, the class of spaces of $0$-cycles \cite{fww}.
\end{abstract}

\section{Introduction}

%%%%%%%%%%%%%%%%%%%%%%%%%%%%%%%%5

Grothendiek rings are  used extensively in algebraic and arithmetic
geometry when addressing enumerative problems involving
algebraic varieties. In this context, if $X$ decomposes into a disjoint union
of locally closed subsets $X = \coprod X_i$ (i.e. a stratification), then in the Grothendieck
ring $K_0(X)$, and by definition, the isomorphism class $[X]$ of $X$ is identified to the sum
$[X] = \sum [X_i]$. The class of the categorical product is defined to be the product in $K_0(X)$ \cite{ neeraja}.

In recent work \cite{kt}, the authors extended the classical construction of the Grothendieck ring $K_0(\mathcal L)$ from the category of algebraic varieties to more general categories $\mathcal L$, whose objects are stratifiable subspaces in Euclidean space, in such a way that strata enjoy a predefined geometrical or topological property, and in such a way also that we retain the fundamental ``motivic morphism'' $\langle\ \rangle : K_0(\mathcal L)\longrightarrow\bbz$ which evaluates to the Euler characteristic with compact supports on locally compact subspaces of finite cohomological type. As it turns out, a suitable such collection $\bbm$ consists of all subspaces $X$ in Euclidean space, which can be decomposed into strata that are ``locally compact of finite type'' with finite type closure conditions (details in \S\ref{prelim}). It is in such a category that we will be working and extending various results from algebraic geometry and topology.

Here's a brief summary of the content of this paper: we will establish  formulas in $K_0(\bbm)$ for the Grothendieck classes of three main geometric constructions that are found and used in the literature. The first formula in Theorem \ref{mainaction} extends, as well as gives a much shortened proof of a formula of Akita \cite{akita} for the class of a quotient of a $G$-space by a discrete action with cocompact quotient and finite stabilizers. As a corollary, and combining with a recent result of \cite{bartosz}, we explicitly express the Euler characteristic of any orbifold quotient by a crystallographic group in terms of the fixed points of the action (Theorem \ref{crystal}).
The second set of results in \S\ref{poly} and \S\ref{poly2} give the class of polyhedral products and polyhedral configuration spaces. These constructions have appeared in various contexts in the literature, and our formula in Proposition \ref{chiiriye} generalizes formulas of \cite{iriye} and \cite{labassi} that were originally expressed in terms of Euler characteristics. Finally, Theorem \ref{fww} gives a very short derivation of a formula of Farb, Wolfson and Wood \cite{fww} on the so-called spaces of $0$-cycles. The derivation in \cite{fww} involves several pages of spectral sequence calculations, and it is stated in terms of Euler characteristics for $0$-cycles of a restricted class of manifolds.

Other results in this paper include a shortened derivation of a formula for the Grothendieck class of ``permutation products'' which are quotients of the form $X^n/\Gamma$, where $\Gamma$ is a subgroup of the symmetric group on
$n$-letters $\mathfrak S_n$, acting on $X^n$ by permutation. This builds on the approach of MacDonald and naturally recovers
his result \cite{macdo}. %Another interesting point we discuss is the existence of a canonical poset associated to(any) $\mathcal M$-stratification of $X$. This yields the existence of a functor from constructible sets to poset-stratified spaces, an observation which seems of independent interest.

As pointed out, a major advantage of our approach is its ease and flexibility. The key advantage that comes out of \cite{kt} and this sequel is that, for stratifiable spaces, it is sometimes more convenient to work with strata directly than it is to work with invariants associated to the strata.

\vskip 10pt
{\sc Acknowledgment}: The first author is grateful to Bartosz Naskr{e}cki for sharing his work and discussing it. The authors extend their appreciation to the Deanship of Scientific Research, Imam Mohammad Ibn Saud Islamic University (IMSIU), Saudi Arabia, for funding this research work through Grant No. (221412033).

%%%%%%%%%%%%%%%%%%%%%%%%%%%%%%%%%%%%%%%%%

\section{Lego Categories}\label{prelim}

We start recollecting some definitions from \cite{kt}. All spaces considered in this paper will be subspaces of Euclidean space $\bbr^n$, for some $n$. They are in particular Hausdorff.

An LC space $X$ in Euclidean space means a ``locally compact'' space.
A space $X$ is LC-stratifiable if $X$ can be partitioned into a \textit{finite} number of strata $X_i$ all of which are LC. It is convenient to write in this case $X=\bigsqcup X_i$ (a disjoint union as sets).  It is also convenient to write a stratification $\mathcal S$ of $X$ as its collection of strata $\mathcal S = \{X_\alpha\}_{\alpha\in I}$ over some indexing set which we again assume to be finite.

We define a ``universal'' collection of stratified spaces
\begin{equation*}\label{principal}
\underline{\mathcal U} = \{(X,\mathcal S)\ |\
X\subset\bbr^n\ \hbox{for some $n$ and
$\mathcal S$ is an LC-stratification of $X$}\}
\end{equation*}
In this notation, an element $(X,\mathcal S)$ is a subspace of $\bbr^n$ with a stratification by LC-strata. So $X$ is LC-stratifiable if $(X,\mathcal S)$ is in $\underline{\mathcal U}$, for some stratification $\mathcal S$.

\bde\label{lego}\cite{kt} A \textit{Lego collection} is any subcollection $\underline{\bbl}$ of $\underline{\mathcal U}$ that is: (i)   closed under cartesian product, and (ii)  closed under removal of strata. This last condition means that if $(X,\{X_j\}_{j\in I})\in\underline{\bbl}$, then
$\displaystyle \left(X\setminus X_i, \{X_j\}_{j\in I\atop j\neq i}\right)\in \underline{\bbl}$.
\ede

We say that $X$ has an $\bbl$-stratification $\mathcal S$ if $(X,\mathcal S)\in\underline{\bbl}$. Alternatively, we say that $X$ is $\bbl$-stratifiable if there is a stratification $\mathcal S$ such that $(X,\mathcal S)\in\underline{\bbl}$.

\iffalse
\bex A most natural way of obtaining Lego collections is as follows. Let $U_n$ be the collection of
all LC-subspaces of $\bbr^n$ and let $U=\bigcup_{n\geq 0} U_n\subset\bbr^\infty$.
Let $L$ be any subcollection of $U$ that is closed under cartesian product. Then
$$\mathcal L = \left\{(X,\{X_i\}))\ |\ X_i\in L\right\}$$
is the \textit{Lego collection of all $L$-stratifiable subspaces}. Below are some examples.
\eex
\fi

Here are the main examples of Lego collections considered in the literature:
\begin{itemize}
\item \textit{(Cell-stratifiable spaces)}. This is the subcollection of all spaces that are stratifiable by open cells, or \textit{cell-stratifiable} spaces (a cell is any space homeomorphic to $\bbr^n$ for some $n$).
A general theory of cell-stratifiable spaces is developed in \cite{tamaki} where regularity conditions are demanded on attachments of cells.
\item\textit{(O-minimal structures)} \cite{coste}. These are subcollections of the cell-stratifiable spaces which also form a Boolean algebra (i.e. closed under finite unions and intersections, and under taking complements). A main example is the collection of real semi-algebraic sets which is the Boolean algebra generated by affine subsets defined by polynomial equalities and inequalities.
\item \textit{(Manifold-stratifiable spaces)} \cite{gz}.
When all strata of a stratification are manifolds, the stratification is  referred to as a ``manifold stratification". In general in the literature, manifold stratifications have more refined conditions like for ``stratifolds'', or requiring the existence of a tubular neighborhood of a stratum, like being the total space of a bundle (Thom, Mather).
\end{itemize}

It is often desirable in the literature that stratifications
$\{X_\alpha\}_{\alpha\in I}$ of a space $X$ verify the ``frontier condition''. In fact, this is sometimes part of the definition of a stratification. By ``Frontier condition'' it is meant that if $X_\alpha\cap\overline{X}_\beta\neq\emptyset$, then $X_\alpha\subset\overline{X}_\beta$. In other words, the strata have the property that the closure of a stratum is a union of other strata. Such stratifications have natural posets associated to them \cite{petersen, tamaki, yokura}, and thus naturally associated spectral sequences as well. The following Lemma shows that there is no loss of generality in starting with any LC-stratification.

\ble Let $\mathcal S = \{X_i\}$ be an LC-stratification of $X$ in $\bbr^n$, for some $n\geq 1$. Then $\mathcal S$ can be canonically refined to a stratification satisfying the frontier condition.
\ele

\begin{proof}
For $A\subset\bbr^n$, with closure $\overline{A}$, we write
$\underline{A}=\overline{A}\setminus A = \overline{A}\cap A^c$.
The set $A$ is locally closed (or equivalently locally compact since working in $\bbr^n$) if and only if $\underline{A}$ is closed. In other words, the part of the boundary of $A$ in $X$ that is in the complement must be closed.
 Note that $\underline{A}$ is written $\check A$ in \cite{allouche}.

Given $\mathcal S = \{X_i\}$ and LC-stratification of $X$, set $Y_i=\mathring{X}_i$ (the interior of $X_i$) and
$Z_i = \overline{X}_i\setminus (\underline{X}_i\cup\mathring{X}_i) = \overline{X}_i\setminus (\mathring{X}_i\cup X_i^c)$ (i.e. this is part of the boundary of $X_i$ in $\bbr^n$ that is in $X_i$). The boundary $\overline{X}_i\setminus\mathring{X}_i$ is closed in $\bbr^n$ thus LC, and $Z_i$ is open in this boundary since $X_i$ is LC. It follows that $Z_i$ is LC by standard properties of LC-spaces. All of $Y_i$ and $Z_i$ are LC, for all $i$, and we can write
$X = \bigsqcup Y_i\sqcup\bigsqcup Z_i$ some refined LC-stratification of $X$ obtained from $\mathcal S$. We must refine this further by decomposing the boundary pieces so that if each part overlaps with a closure $\overline{X_j}$, $j\neq i$, then it is completely contained in that closure. This is done by writing
$$Z_i = \left(Z_i\setminus \bigcup_{i\neq j}\overline{X}_j\right)\sqcup\bigsqcup_{j\neq i}
\left(Z_i\cap \overline{X}_j\setminus\bigcup_{k\not\in\{i,j\}}\overline{X}_k\right)\sqcup\cdots\sqcup \left(Z_i\cap\bigcap_{j\neq i}\overline{X}_j\right)$$
The resulting stratification of $X$ can now be checked to verify
the frontier condition.
\end{proof}

%%%%%%%%%%%%%%%%%%%%%%%%%%%%%%%%%%%%%%%%%%%%

\subsection{Grothendieck Rings and Motivic Morphisms}
Let $\underline{\bbl}$ be a Lego collection (Definition \ref{lego}). We define
the lego category $\mathcal L$ to be the category whose objects are all spaces $X$ such that $(X,S)\in\underline{\bbl}$ for some $S$.  In other words, the objects of the category are the $\bbl$-stratifiable subspaces. Morphisms are the continuous maps. This is a small category, and
the isomorphism classes form a set.
We can then define the commutative ring
$K_0(\mathcal L)$
to be the quotient of the free abelian group on isomorphism classes $[X]$ of objects $X\in\mathcal L$ by the subgroup generated by $[X]-[Y]-[X\setminus Y]$, whenever $Y$ is a stratum of an $\bbl$-stratification of $X$, and by  $[X][Y]-[X\times Y]$.
The zero element of the ring is $[\emptyset]$ and the identity for multiplication is the class of a point $[pt]$.

By definition, in $K_0(\bbl )$, if $X=\bigsqcup X_i$ is an $\bbl$-stratification, then $[X]= \sum [X_i]$.

Unitary ring morphisms  $\langle\ \rangle : K_0(\mathcal L)\longrightarrow R$, where $R$ is a commutative ring, are called ``motivic morphisms''.
It is not clear in general that an interesting motivic morphism exists for a given $\bbl$, meaning that it is well-defined and non-trivial. Any such morphism evaluated on $X\in\mathcal L$ must be independent of the stratification used, up to isomorphisms in the category.

In particular, a Lego category $\mathcal L$ becomes of interest, at least in the viewpoint of \cite{kt}, if the ring morphism $\langle\ \rangle : K_0(\mathcal L)\rightarrow\bbz$, sending $[X]=\sum [X_i]$ to $\sum\chi_c(X_i)$, is a well-defined morphism. We give some examples to explain why the construction of such morphisms is subtle and non-trivial.

An LC space $X$ is of finite cohomological type (i.e. cft) if
$H^*_c(X,\bbz)$ vanishes beyond a certain degree, and is finitely generated otherwise,  see \eqref{hc}.
We say a space is ``LCFT'' if it is LC, cft and locally contractible. Then $\chi_c$ is the Euler characteristic with compact supports defined for LCFT spaces by
$$\chi_c(X) = \sum (-1)^i\hbox{rank} H^*_c(X,\bbz)\in\bbz$$
It is well-known in the literature that $\chi_c$ is ``combinatorial'' meaning that it is additive on suitable stratifications of LC-spaces. More precisely, suppose $X$ is LC and locally contractible, and let
$\{X_i\}$ be an LCFT stratification of $X$. Then
$$
\chi_c(X) = \sum \chi_c(X_i)$$
This statement is often used in complex geometry where various subtleties do not occur. When working in Euclidean space, extra care must be taken.

The additivity of $\chi_c$ is a consequence of the long exact sequence in cohomology with compact supports. This sequence is exact because $X$ is both LC and HLC ``homologically locally connected'' (\cite{bredon},\S1). The HLC condition is ensured if the space if locally contractible, and this is what we require.
If we relax the conditions on $X$ of being LC or HLC, then $\chi_c$ might fail to be additive. The following two examples show that the two hypotheses on $X$ are necessary.

\bex \cite{kt}
Consider the spaces $X$ and $Y$ as depicted in Figure \ref{square}, with
 $\langle X\setminus Y\rangle = \langle X\rangle - \langle Y\rangle = 1-(-1)=2$.
It can be checked that $H^*_c(X\setminus Y)=0$ using
the definition of cohomology of compact supports
\begin{equation}\label{hc}
H^i_c(X) = \lim_{K\subset X} H^i(X,X\setminus K)
\end{equation}
with the limit running over the directed system of homomorphism $H^i(X,X\setminus K_1)\rightarrow H^i(X,X\setminus K_2)$ induced from inclusions of compact sets $K_1\hookrightarrow K_2$.
In our case, any compact $K$ in $X\setminus Y$ is contained in a compact $K'$ with contractible complement, thus
$H^*_c(X\setminus Y)=0$ in positive degrees. This shows that
$\langle X\setminus Y\rangle\neq \chi_c(X\setminus Y)$.
\begin{figure}[htb]
\begin{center}
\epsfig{file=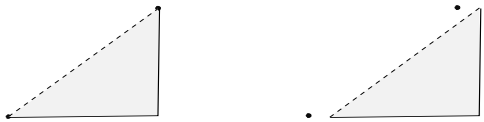,height=0.6in,width=2.5in,angle=0.0}
\caption{
An example of a non-LC space (left) decomposable into LC-strata (right): $X$ is the closed triangle and $Y\cong (0,1)$ is the diagonal segment without its endpoints, both in $\bbr^2$. Then $X\setminus Y$ (left figure) is not LC but it is LC-stratified as $X\setminus Y = S^0\sqcup (X\setminus \overline{Y})$ (right figure), where $\overline{Y}$ is the closure of $Y$.}\label{square}
\end{center}
\end{figure}
\eex

\bex Let $X = \displaystyle\left\{\left(x,\sin \left({1\over x}\right)\right), x\in (0,1]\right\}\cup
\{0\}\times [-1,1]$ be the closure of the topologist's sine curve.
This space is compact but not HLC. A computation shows that
$$H_0(X,\bbz) = \bbz\oplus\bbz\ \ ,\ \ H_i(X,\bbz)=0, i\geq 2$$
which means that $\chi (X) = \chi_c(X)=2$. Consider the cell-stratification
$$\{(1,\sin 1)\}\sqcup \{(0,1)\}\sqcup \{(0,-1)\}\sqcup \left\{\left(x,\sin \left({1\over x}\right)\right), 0<x<1\right\}\sqcup \{(0,t), -1<t<1\} $$
This gives that
$\langle X\rangle = 3-2=1$, which does not correspond to $\chi_c(X)$.
\eex

The above two examples show that the value
$\langle X\rangle := \sum \chi_c(X_i)$ is not always a characteristic if $X$ is not LC or not locally contractible.
The bulk of the work in \cite{kt} was isolating some key geometrical properties in $\bbl$ so that $\langle\ \rangle$ exists. The main collection of interest is the following.
\begin{eqnarray}\label{bbm}
\underline{\bbm} &:=&\left\{(X,\mathcal S)\in\underline{\mathcal U}\ \
\hbox{with $\mathcal S= \{X_i\}$ being such that all possible unions over closures}\ \bigcup\overline{X}_i\right.\nonumber\\
&&\ \ \ \left.\hbox{are cft and locally contractible}\right\}
\end{eqnarray}
As before, if $(X,\mathcal S)\in\underline{\bbm}$, then we say that $\mathcal S$ is an $\bbm$-stratification of $X$, and we associate the category $\mathcal M$ of all $\bbm$-stratifiable spaces.

The main result  below shows that $K_0(\bbm)$ is a non-trivial and a pretty useful object of study.

\begin{theorem} \cite{kt} The morphism
$$\langle-\rangle : K_0(\bbm)\lrar \bbz,\ \
[X]\longmapsto \langle X\rangle = \sum \chi_c(X_i),\ \hbox{if}\ \{X_i\}\ \hbox{an $\bbm$-stratification of $X$}$$
is a  well-defined ring morphism, meaning it neither depends on the isomorphism class of $X$ nor does it depend on the
$\bbm$-stratification.
\end{theorem}

In this paper and throughout, we assume all spaces to be $\bbm$-stratifiable, and Grothendieck classes will be taken in $K_0(\bbm)$. We say $X$ is $\bbm$-stratified if we fix an $\bbm$-stratification \eqref{bbm} of $X$.

\iffalse
\bre It is asserted in \cite{johnson} that if $f: X\rightarrow Y$ is a simplicial map between finite simplical complexes such that for every $y\in Y$, $\chi (f^{-1}(y))=k$, then $\chi (X) = k\chi (Y)$. This property is no longer true if we replace simplicial by CW.
\ere
\fi

%%%%%%%%%%%%%%%%%%%%%%%%%%%%%%%%%%%%%%%%%%%%%%%%%%%%%%%%%%
\section{The Class of Cocompact Group Actions}

%https://ncatlab.org/nlab/show/one-point+compactification

We streamline the notion of a stratified group action and give formulas for computing the Grothendieck class of a quotient.
We start with the following well-known classical Burnside type formula for finite groups (see Proposition 7 of \cite{gz}).
%https://stacks.math.columbia.edu/tag/04ZC

\iffalse
\vskip 5pt
{\color{blue} A demontrer (je pense que c'est vrai):
If $E\in\bbm$ and a finite $G$ acts locally linearly on $E$, then $X = E/G\in\bbm$. If the action is free, then $[E]=|G|[X]$. (see \cite{zimmer})}.
The proof should use the fact that if the action is locally linear on X LCFT, then orbit stratification is of finite type.

\ble\label{multiplicative} The collection $\mathcal M$ is closed under taking joins $X*Y$ and one has
\begin{equation*}
[X*Y] = [X]+ [Y]-[X][Y]
\end{equation*}
\ele

\begin{proof}
The second claim is a  computation.
Let $CX$ be the cone with base $X$, and write in $\mathcal L$
$$X*Y = (CX\times Y)\times_{X\times Y}(X\times CY)$$
This gives $[X*Y]= [CX][Y]+[X][CY]-[X][Y]$.  It now suffices to see that $[CX]=[pt]$ for any $X\in\bbm$. But
$CX$ decomposes as
$(X\times D^1)\sqcup X\sqcup pt$, with $D^1=]0,1[\cong\bbr$,
so that $[CX]=[X]+[X][D^1] + [pt]
= [X]-[X]+[pt]=[pt]$, as desired.
\end{proof}
\fi

\ble\label{burnside} Let
$X^g$ denote the fixed points of the subgroup generated by $g$. Then
\begin{equation}\label{burnside}
\left[X/G\right] = {1\over |G|}\sum_{g\in G}\left[X^{g}\right]
\end{equation}
\ele

We include a proof below  to illustrate the (nice) fact that this statement is simply a ``step-by-step categorification'' of the standard proof of the Burnside counting lemma (\cite{tomdieck}, p.225).

\begin{proof}
Consider the subset
$W=\{(g,x)\in G\times X\ |\ gx = x\}= \bigsqcup_{g\in G}\{g\}\times X^g$. By the multiplicative property of the Grothendieck ring, $\displaystyle [W]=\sum_{g\in G}[\{g\}\times X^g]= \sum_{g\in G}[X^g]$.
On the other hand, the projection
$\pi: W\rightarrow X$ has fiber over $x$ the stabilizer $G_x$.  Stratify $X$ by orbit type, $X=\bigsqcup X_{(H)}$, where
$$X_{(H)}=\{x\in X\ |\ G_x\ \hbox{is conjugate to}\ H\}$$
By construction $\pi^{-1}(x)=G_x$ is conjugate to $H$, if $x\in X_{(H)}$. But $\pi$ restricted to $\pi^{-1}(X_{(H)})$ is a covering of $X_{(H)}$ of
degree $|H|$,  thus
$[\pi^{-1}(X_{(H)})] = |H|[X_{(H)}]$.
Putting these observations together
\begin{eqnarray*}
\sum_{g\in G}[X^g]&=&
[W]=[\pi^{-1}(X)]
=\sum_{(H)}[\pi^{-1}(X_{(H)})]
=\sum_{(H)}|H|[X_{(H)}]
=\sum_{(H)}|G|\frac{[X_{(H)}]}{|G/H|}\\
&=&|G|\sum_{(H)}[X_{(H)}/G]=|G|[X/G]
\end{eqnarray*}
from which the result follows after dividing by $|G|$.
\end{proof}

We now seek a formula for $[X/G]$ when $G$ is infinite acting with compact quotient, as in the case of crystallographic groups for example. The Burnside formula \eqref{burnside} does not hold  for obvious reasons. Akita in \cite{akita} gave an alternate formula, which works at the level of Euler characteristics, when $G$ is discrete acting cellularly with compact stabilizers, and having a compact quotient. We here give a similar formula for $[X/G]$, for the same type of action, which streamlines Akita's result.

\bde \label{stratact}
Let $X$ be LC-stratified and $G$ a discrete group acting on $X$. We say the action is stratified, or $X$ is $G$-stratified, if $G$ takes a stratum to a stratum, and if whenever $\gamma\in G$ fixes a point of a stratum, then it fixes the whole stratum.
A special case is when $X$ is a CW-complex stratified by its cells, in which case this action is generally referred to as a $G$-CW action.
\ede

\bre If $X$ is $G$-stratified,  then $X^g$ is a $G$-stratified subspace by the strata of $X$ that intersect it.
Indeed,let $S$ be a stratum of $X$ so that $S\cap X^g\neq\emptyset$. Then $S\subset X^g$ by the second condition in the definition \ref{stratact}.
\ere

\bde A $G$-stratified action on an LC-stratified space $X$ is of \textit{finite type} if the following two properties hold: \begin{enumerate}[label=(\roman*)]
\item The quotient $X/G$ is $\bbm$-stratified, with a finite number of strata, and

\item The isotropy subgroup $G_S$ of each stratum $S$ is of finite order.
\end{enumerate}
\ede

\iffalse
{\color{blue} Verifier qu'une telle action a a un domaine fondamentale. Par exemple l'action du groupe cyclique $\begin{pmatrix} 0&1/2\\ 2&0\end{pmatrix}$ agissant on the upper half plane $y>0$ discretement mais n'a pas de domaine?}
\fi

We denote $|G|$ the order of the group $G$, $F(G)$ the set of representatives of conjugacy classes of elements of finite order of $G$ and $C(g) = \{h\in G| hg=gh\}$  the centralizer of $g\in G$. Then $C(g)$ acts on the fixed point set $X^g$  under $g$ and this action is stratified. As is common, we denote by
$\Gamma_S$ the isotropy subgroup of the stratum $S$ under the action of $G$. The following is the main theorem of this section.

\begin{theorem}\label{mainaction} Let $G$ be a finite type stratified action on an LC-stratified $X$, and let $C_{\Gamma_{ S}}(g) = C(g)\cap\Gamma_S$ (that is the subgroup of elements that commute with $g$ and fix $S$). Let $\mathcal E_g$ be the set of representatives of $C(g)$-orbits of strata in $X^g$. Then
$$[X/G] = \sum_{g\in F(G)}\sum_{S\in\mathcal E_g}{[S]\over |C_{\Gamma_{ S}}(g)|}$$
\end{theorem}

\begin{proof}
Let $S$ be a representative of an orbit in $X/G$, $\Gamma_S$ its isotropy group, $F(\Gamma_S)$ the set of the representatives of the conjugacy classes in $\Gamma_S$ (modulo $\Gamma_S$) ,
and $C_{\Gamma_S}(g)=C(g)\cap \Gamma_S$. We can write that :
\begin{eqnarray*}
[S]&=&\sum_{g\in \Gamma_S}\frac{[S]}{|\Gamma_S|}
= \sum_{g\in F(\Gamma_S)}\frac{[S]}{|\Gamma_S|}\frac{|\Gamma_S|}{|C_{\Gamma_S}(g)|}
= \sum_{g\in F(\Gamma_S)}\frac{[S]}{|C_{\Gamma_S}(g)|}
\end{eqnarray*}
where $\displaystyle\frac{|\Gamma_S|}{|C_{\Gamma_S}(g)|}$ is the number of elements in the conjugacy class of $g$.

For $g\in F(G)\cap \Gamma_S$, consider the set $T_{S}^{g}=\{\gamma \in G | \gamma^{-1}g\gamma \in \Gamma_S\}$.
We have that
$$ \Gamma_S=\coprod_{g\in F(G)\cap \Gamma_S}\left\{\gamma^{-1}g\gamma \,\,;\,\,\gamma\in T_{S}^{g}\right\}$$
Notice that for $\gamma_1,\gamma_2 \in T_{S}^{g}$, the following holds
\begin{center}
$\gamma_{1}^{-1}g\gamma_{1}$ is conjugate to $\gamma_{2}^{-1}g\gamma_{2}$ in $\Gamma_S$
$\Longleftrightarrow$ $\gamma_{2}\in C(g).\gamma_{1}.\Gamma_S$
\end{center}
Consequently, and as in \cite{brown}, we can consider the double quotient
 $$U_{S}^{g}=C(g)\backslash T_{S}^{g}/\Gamma_S $$
In there, $\{\gamma^{-1}g\gamma, \gamma \in U_{S}^{g}\}$ is a set of representatives of the different conjugacy classes in $\Gamma_S$ (mod $\Gamma_S$) which are conjugate to $g$ in $G$. We incorporate this in our earlier expression for $[S]$
\begin{eqnarray*}
[S]&=&\sum_{g\in F(G)\cap\Gamma_S}\sum_{\gamma\in U_{S}^{g}}\frac{[ S]}{|C_{\Gamma_S}(\gamma^{-1}g\gamma)|}=\sum_{g\in F(G)\cap\Gamma_S}\sum_{\gamma\in U_{S}^{g}}\frac{[\gamma S]}{|C_{\Gamma_{\gamma S}}(g)|}
\end{eqnarray*}
The last equality follows from the fact that
$$ C_{\Gamma_S}(\gamma^{-1}g\gamma)=\Gamma_{\gamma S}\cap C(g)=C_{\Gamma_{\gamma S}}(g)$$
Next, and for $\gamma \in U_{S}^{g}$, we have
\begin{eqnarray*}
\left\{hS, h\in C(g).\gamma .\Gamma_S\right\} = \left\{c.\gamma. s, c\in C(g), s\in\Gamma_S\right\}
= \left\{c \gamma S, c\in C(g)\right\}
\end{eqnarray*}
and therefore $\{\gamma S\,,\, \gamma\in U_{S}^{g}\}$ is a set of representatives of the orbits of $X^g$ modulo $C(g)$. This implies that
\begin{eqnarray*}
[S]&=&\sum_{g\in F(G)\cap\Gamma_S}\sum_{S\in\mathcal E_g}{[S]\over |C_{\Gamma_{ S}}(g)|}
\end{eqnarray*}
At last, if  $\mathcal E$ is a set of representatives in $X/G$, we can write
\begin{eqnarray*}
[X/G]&=&\sum_{S\in \mathcal E}[S] =\sum_{g\in F(G)}\sum_{S\in\mathcal E_g}{[S]\over |C_{\Gamma_{ S}}(g)|}
\end{eqnarray*}
which is our desired expression.
Notice that the proof is independent of the choice of the representative $S$ in $X/G$, since if $S'$ is in the orbit of $S$, its isotropy group is conjugate to that of $S$.
\end{proof}

\bex
Let $X=\bbr$ and $G$ the ``infinite dihedral group'' $\left\langle \tau,g \,|\, \tau^2=1, \tau g \tau=g^{-1}\right\rangle$, with $g,\tau$ acting as follows:
$\tau(x)=-x$ and $g(x)=x+1$. Here $\tau$ is the central reflection centered at the origin, and $g$ is translation to the right by $1$. Observe that $g^n \tau$ is also a reflection (so of order $2$) with a unique fixed point $\frac{n}{2}$. Stratify $\bbr$ as follows
\begin{eqnarray*}
\mathbb{R}=\bigcup_{n\in \mathbb{Z}} \left\{\frac{n}{2}\right\}\,\sqcup \,\left]\frac{n}{2},\frac{n+1}{2}\right[
\end{eqnarray*}
Then the action of $G$ is stratified with respect to this stratification.
It is clear that $X/G=[0,1/2]$ so $[X/G]= [pt]=1$.

We can now try to apply Theorem \ref{mainaction} to this situation. We notice that the elements of finite order are those of the form $r_n = g^n\tau$ (the reflections), and that these are conjugate according to the parity of $n$, for we can check that $g^{-1}r_ng = r_{n-2}$. Consequently, the set of elements of finite order up to conjugation is $F(G)=\{id,\tau, g\tau\}$, and the fixed-point sets are $\bbr^{id}=\mathbb{R},\
 \bbr^{\tau}=\{0\},\ \bbr^{g\tau}=\{1/2\}$.
On the other hand, the centralizer $C(id)$ is the whole group, and $X^{id}=\bbr$, so all the strata are in the orbits of $C(id)$ acting on $\bbr$. But there can only be three orbits, namely that of $\{0\}$ (giving all the integers), that of $\left\{{1\over 2}\right\}$ giving all midpoints, and that of $]0,{1\over 2}[$. This means that $\displaystyle \mathcal E_{id}=\left\{\{0\},]0,1[, \{1/2\}\right\}$.

Similarly, we verify that
$\mathcal E_{\tau}=\{0\}$ and that
$\mathcal E_{g\tau}=\{1/2\}$.

For every stratum $S\in\mathcal E(g)$, $g\in F(G)$, we need find $C_{\Gamma_S}(g) = \Gamma_S\cap C(g)$. They are given below.
\begin{itemize}
\item  $g=id$, $C({id})= G$. There
are three strata in $\mathcal E(id)$ and for each we find: (i) $\Gamma_{]0,1[} = \{id\}\, \hbox{and} \,
C_{\Gamma_{]0,1[}}(id) = \{id\}$, (ii)
$\Gamma_{\{0\}}=\{id,\tau\}\,\hbox{and} \,
C_{\Gamma_{\{0\}}}(id) = \{id,\tau\}$, (iii)
$\Gamma_{\{1/2\}}=\{id,g\tau\}$ and
$C_{\Gamma_{\{1/2\}}}(id) = \{id,g\tau\}$.
\item $g=\tau$, $C(\tau) = \{id, \tau\}$, and there is only one stratum $\{0\}$, with $\Gamma_{\{0\}} = \{id, \tau\}$ and
$C_{\Gamma_{\{0\}}}(id) = \{id,\tau\}$,
\item $g=g\tau$, $C(g\tau) = \{id, g\tau\}$. There is only one stratum $\{1/2\}$ in $\mathcal E_{g\tau}$. Again
$\Gamma_{\{1/2\}}= \{id,g\tau\}$ and
$C_{\Gamma_{\{1/2\}}}(g\tau ) = \{id,g\tau\}$.
\end{itemize}

We can now compute
\begin{equation}\label{compute}
[X/G]= \left(\frac{[0]}{2}+[]0,1[]+\frac{[1/2]}{2}\right)+\frac{[0]}{2}+\frac{[1/2]}{2}= {1\over 2} - 1 + {1\over 2} + {1\over 2} +{1\over 2} = 1
\end{equation}
This is in agreement with our earlier calculation.
\eex
\bre Note that in \eqref{compute} of the previous example, the quantity $\left(\frac{[0]}{2}+[]0,1[]+\frac{[1/2]}{2}\right)$ is precisely the orbifold Euler characteristic of $G$ acting on $X=\bbr$ (as a crystallographic group) and this quantity is always $0$ indeed \cite{bartosz}. More about this in \S\ref{crystal}.\ere
%%%%%%%%%%%%%%%%%%

\section{Applications}

Theorem \ref{mainaction} and its corollary, Burnside's formula, have compelling applications. We develop here three of those.

\subsection{Akita's Formula \cite{akita}}\label{akita}
When $X$ is a CW complex, stratified by its cells, then a discrete group $\Gamma$ has a stratified action if and only if it has a CW-action. This means $\Gamma$ permutes cells and has finite stabilizers. Akita \cite{akita} calls such an action $\Gamma$-finite if the quotient is a finite complex. In this case, he defines the \textit{orbifold Euler characteristic}
\begin{equation}\label{orbeuler}
e(\Gamma, X) = \sum_{\sigma\in\mathcal E}(-1)^{\dim\sigma}\frac{1}{|\Gamma_\sigma|}\in \bbq
\end{equation}
where $\mathcal E$ is a set of representatives of $\Gamma$-orbits of cells of $X$ and $\Gamma_\alpha$ is the stabilizer of a cell. One sets $e(\Gamma,\emptyset)=0$. It is to be noted that this quantity does not depend on the cell decomposition of $X$.
For each $\gamma\in\Gamma$, write the centralizer $C_\Gamma (\gamma)$.

\bco\label{akitatheorem} (Akita \cite{akita}) Let $\Gamma$ be a discrete group, and $X$ a $\Gamma$-finite CW-complex. Then
$$\chi (X/\Gamma) = \sum_{\gamma\in {\mathcal F}(\Gamma)}e(C_{\Gamma}(\gamma),X^{\gamma})$$
where $\mathcal F(\Gamma)$ is a set of representatives of conjugacy classes of elements of finite order in $\Gamma$.
\eco

\begin{proof}
This is an immediate application of the motivic morphism $\langle\ \rangle$ applied to our formula in Theorem \ref{mainaction}. In this case $\langle X/\Gamma\rangle = \chi (X/\Gamma)$ since $X/\Gamma$ is compact, and $\langle S\rangle =\langle \sigma\rangle = (-1)^{\dim\sigma}$.
\end{proof}

The proof of Corollary \ref{akitatheorem} is in striking contrast with the proof given by Akita which takes many pages and involves the use of equivariant topology, homological algebra and spectral sequences. Clearly however, Theorem \ref{mainaction} was totally motivated by the work of Akita.

\subsection{Permutation Products \cite{macdo}}\label{permprod}

Let $G$ be a subgroup of the symmetric group on $n$-letter $\mathfrak S_n$, and consider the quotient $GP^n(X) := X^n/G$, with $G$ acting by permutation of entries. Note that if $X\in\bbm$, then $GP^n(X)\in\bbm$ \cite{kt}. The special case of $G=\mathfrak S_n$ gives the symmetric products, commonly denoted  by $\sp{n}(X)$, which are a key construction in both algebraic geometry and topology. In \cite{macdo}, I.G. MacDonald gave a formula for the Poincar\'e series of these spaces. We here recover his formula for the Euler characteristic by a shortened argument in $K_0(\bbm)$.

To express $[GP^n(X)]$ in terms of the invariants of $X$, we should observe first that this class depends strongly on the action of $G$, and not just on the abstract $G$. To see this, consider the subgroup $\bbz_2$ embedded in $\mathfrak S_4$ as either $G_1=\{e, (12)\}$ or as $G_2=\{(e, (12)(34)\}$. Then
$X^4/G_1 = \sp{2}(X)\times X^2$ while $X^4/G_2=\sp{2}(X\times X)$, which of course
are different spaces with two different classes.
What turns out to capture this difference when it comes to the Euler characteristic, or the Grothendieck class, is precisely the action of the group on its left cosets in $\mathfrak S_n$.

Let's recall some definitions.
Every permutation in $\mathfrak S_n=\hbox{Perm}\{1,\ldots, n\}$ has a ``disjoint cycle decomposition'' or ``cycle type'' which corresponds to
a product of disjoint $\lambda_i$-cycles.
To a cycle type we can then associate the unordered partition of $n$; $\lambda= \lambda_1+\lambda_2+\ldots +\lambda_k$ with length $|\lambda | = k$. The length $k$ is the number of disjoint cycles in the decomposition.

Now $G$ being a subgroup of $\mathfrak S_n$, then $G$ acts on its set of left cosets $\mathfrak S_n/G$ by left multiplication. We set then
$$\chi^G(\sigma)= |\left(\mathfrak S_n/G\right)^\sigma|$$
This is the number of left cosets fixed by $\sigma$.
This number only depends on the conjugacy class of $\sigma$, so in particular only depends on the cycle type decomposition of $\sigma$.
We can then define $\chi_{\lambda}^G$ to be $\chi^G(\sigma)$ where $\sigma$ is any permutation whose cycle decomposition corresponds to $\lambda$.
Note that
$\chi^G(\sigma)=0$ if $\sigma \in S_n$ has no conjugate in the subgroup $G$.

\bpr Let $G$ be a subgroup of $S_n$ acting on $X^n$ by permutations. Then
$$[GP^n(X)]=\frac{1}{n!}\sum_{\lambda}h_{\lambda}\chi_{\lambda}^G [X]^k$$
where $\lambda=(n_1\ldots,n_k)$, $n_1\geq n_2\geq\cdots\geq n_k$, is a partition of $n$, and $h_\lambda$ is the number of permutations of $\mathfrak S_n$ with cycle type $\lambda$.
\epr

\begin{proof}
We write $(X^n)^\sigma=Fix(\sigma)$. Note that $[F(t_j\sigma t_j^{-1})] = [F(\sigma )]$ since the fixed point sets are homeomorphic via the map
$F(\sigma )\rightarrow F(t_j\sigma t_j^{-1}), x\mapsto t_jx$. The idea of the proof next is to use the fact that $F$ behaves like a character (i.e. it is constant on conjugacy classes) so we can apply an analog of Lemma 6 of \cite{macdo}.

Suppose the partition of $n$ associated to the cycle type of $\sigma$ is $\lambda$, then $Fix(\sigma)\cong X^{|\lambda|}$.
Let $r$ be the index of $G$ in $\mathfrak S_n$, so that $n! = r|G|$. We start by rewriting Burnside's formula \eqref{burnside}
\begin{equation}\label{first}
[GP^n(X)]=\frac{1}{|G|}\sum_{\tau\in G} [(X^n)^\tau]
=\frac{1}{n!}\sum_{\tau\in G} r. [Fix(\tau)]
\end{equation}
List the left cosets as in $t_1G,\ldots, t_rG$. If $\sigma$ is
any element of $\mathfrak S_n$, $\chi^G(\sigma)$ is the number of suffixes j for which $\sigma t_jG = t_jG$ i.e. for which $\sigma\in G^{t_j}$.
Now pick $\sigma\in\mathfrak S_n$. If $\sigma$ is not conjugate to any $g\in G$, then $\chi^G(\sigma) = 0$. If it is, then
$t_j\sigma t_j^{-1} = \tau_j\in G$ for some $j$, or $\sigma = t_j^{-1}\tau_j t_j$, and there are
$\chi^G(\sigma)$-as many such $j$'s.
This immediately implies the second equality below
\begin{equation}\label{second}
\sum_{\tau\in G} r. [Fix(\sigma)]=
\sum_{\tau\in G}\sum_{j=1}^{r} [Fix(t_{j}^{-1}\tau t_{j})] =
\sum_{\sigma\in S_n}\chi^G(\sigma)[Fix(\sigma)]
\end{equation}
By combining \eqref{second} with \eqref{first} we obtain the desired formula
\end{proof}

\iffalse
\begin{proof}
If $H=gGg^{-1}$ for some $g\in S_n$ then the left cosets of $G$ and $H$ in $S_n$ are respectively given by
$$S_n/G=\{t_1G,t_2G,\ldots,t_rG\}\;\;,\;\;S_n/H=\{gt_1g^{-1}H,gt_2g^{-1}H,\ldots,gt_rg^{-1}H\}$$
Hence, for all $\sigma\in S_n$, $\chi^G(\sigma)=\chi^H\sigma)$ since the equality $k.t_iG=t_iG$ is equivalent to $gkg^{-1}.(gt_ig^{-1}H)=gt_ig^{-1}H$.
\end{proof}
\fi

\bex When $G=\bbz_n$ is the cylic group on $n$-letters. The quotient $GP^n(X)=\cp{n}(X)$ is the so-called cyclic product. Let $\bbz_n$ be generated by $\sigma$ of order $n$.
The cycle type of $\sigma$ is $n$ (one block-partition). If $k$ is an integer and $d=gcd(k,n)$, then the cycle type of $\sigma^k$ is ${n\over d}+\cdots + {n\over d}$ ($d$-times), and the number of permutations in $G$ having this cycle type is $\phi(d)$, where $\phi$ is Euler's function. Consequently
(see \cite{macdo, kt})
$$\displaystyle [CP^n(X)] = \frac{1}{n}\sum_{d|n}\phi(d)[X]^{\frac{n}{d}}$$
\eex

\subsection{Crystallographic Groups} \label{crystal}

This section is motivated by the recent article \cite{bartosz}.
Let $Isom(n)$ be the group of affine isometries of $\bbr^n$. As is known, this is a semi-direct product of $O_n(\bbr)$ and the translation group $T\cong \bbr^n$. By definition, a Euclidean Crystallographic Group  (or ``ECG'') is a discrete subgroup of $Isom(n)$ having compact quotient. Being discrete means that no orbit of the action has an accumulation point.
If $\Gamma$ is an ECG, then $\Gamma$ acts properly discontinuously on $\bbr^n$ and the quotient has the structure of an orbifold. We use
\cite{farkas} as our basic reference on these groups (see also  \cite{burde}). The following main result is needed. For the definition of the orbifold Euler characteristic, see \eqref{orbeuler}.

\begin{theorem} (Theorem 4.1 \cite{bartosz})
The orbifold $\bbr^N/\Gamma$, obtained from the action of the crystallographic group $\Gamma$ on $\bbr^N$ for $N\geq 1$,
has an orbifold Euler characteristic equal to $0$.
\end{theorem}

By combining this theorem with the formula of Akita, we will derive a formula the classical Euler characteristic of the  crystallographic quotient $\bbr^n/G$. We know of no reference to this result in the literature. As before, we write
$C_\Gamma (\gamma)$ the centralizer of $\gamma$ in $\Gamma$.

We call a \textit{central-isometry} any isometry $\gamma$ such that $X^\gamma$ is reduced to a single point. This is generally given by a rotation composed with a reflection whose hyperplane is orthogonal to the axis of rotation, at least in dimension 3. The algebraic characterization of such isometries is as follows: write any such transformation $F$ as a composition $F =  tr_{v}\circ\phi$, where $tr_v$ is translation by a vector, and $\phi$ is a linear transformation (i.e. \textit{the linear part of $F$}). Then such an $F$ has a unique fixed point if and only if $\phi - Id$ is an isomorphism of $\bbr^n$ \cite{ryan}.

%file:///C:/Users/user/Downloads/On_the_fixed_points_of_an_affine_transformation_an.pdf

%https://hal.archives-ouvertes.fr/hal-03665008/document

\begin{theorem}\label{crystal} Let $\Gamma$ be an ECG acting on $\bbr^n$, $n\geq 1$. If $\Gamma$ contains no central isometries (i.e. elements with a single fixed point), then $\chi (\bbr^n/\Gamma) = 0$. Otherwise, let $\mathcal R(\Gamma)\neq\emptyset$ be the set of representatives of conjugacy classes of all central isometries.  Then $$\chi \left(\bbr^n/\Gamma\right) = \sum_{\gamma\in \mathcal R(\Gamma)}{1\over |C_\Gamma (\gamma)|}$$
Moreover, $[\bbr^n/\Gamma] = \chi (\bbr^n/\Gamma)[pt]$.
\end{theorem}

\begin{proof}
We will be using Akita's formula (Corollary \ref{akita}). To that end, we need to argue that $\bbr^n$ is a $\Gamma$-finite CW-complex (see \S\ref{akita}). This is a consequence of a construction attributed to
Dirichlet and Voronoi of a compact fundamental domain which is a polytope and which is then propagated by the
action of $\Gamma$. For more general results on fundamental domain, see Illman (\cite{illman}, part II).

We need understand $X^\gamma$ for $\gamma\in \Gamma$. The fixed points of elements in the isometry group of $\bbr^n$ are affine subspaces of $\bbr^n$. Moreover, the action of $C_\Gamma (\gamma )$ on $X^\gamma$ is also an ECG action, and is compatible with the cell stratification.
So whenever $X^\gamma$ is not reduced to a point, each term
$e(C_{\Gamma}(\gamma),X^{\gamma})$ is an orbifold characteristic of an ECG, and thus is zero by Theorem 4.2 of \cite{bartosz}. When $X^\gamma$ is reduced to point, then $ e(C_\Gamma(\gamma), X^\gamma) = {1\over |Stab_{X^\gamma}|}$, where
the denominator is the cardinality of the stabilizer of the unique fixed point of $\gamma$, in the group $C_\Gamma(\gamma)$. This is clearly the entire group, since any element that commutes with $\gamma$ leaves  $X^\gamma$ invariant, and since $X^\gamma$ consists of a single point, it must fix it. So $\displaystyle {1\over |Stab_{X^\gamma}|} = {1\over |C_\Gamma(\gamma)|}$. Now we take the sum over all such $\gamma$ running over representatives of conjugacy classes (Corollary \ref{akita}).

Finally, since $\bbr^n$ is a $\Gamma$-finite cell complex, which means that for each cell $S$, $[S]= (-1)^{\dim S}[pt]$, we see that $[\bbr^n/\Gamma] = \chi_c(\bbr^n/\Gamma)[pt]$. But $\bbr^n/\Gamma$ is compact,  so we can replace $\chi_c$ by $\chi$ in this case.
\end{proof}

%%%%%%%%%%%%%%%%%%%%%%%%%%%%%%%%%%%%%%

%\section{Polyhedral Products and Configurations}\label{diagarrang}

%Polyhedral products are well-known and very well-studied. “Polyhedral configuration spaces” which we define here, have only been considered in a handful of references [7, 17, 19] and should form an interesting subject of research. In this section, we compute the Grothendieck classes of both families of spaces.

%____________________________________________________________

\section{Polyhedral products}\label{poly}

Polyhedral products are a special family of simplicial complexes which have been extensively studied in recent years. To define them, start with $K$ an abstract simplicial complex on the vertex set $\Omega:= \{1,\ldots, n\}$. Given a pair of pointed spaces $(X,A)$, and a simplex $\sigma\in K$, $\sigma = [i_1,\ldots, i_k]$, define the subspace of $X^n$
\begin{equation}\label{poly1}
(X,A)^\sigma = \prod_{i=1}^n Y_i\ \ \ \hbox{where}\ \ \ Y_i=\begin{cases}X& \hbox{if}\
i\in\sigma\\
A&\hbox{if}\ i\not\in\sigma\end{cases}
\end{equation}
Note that with this definition, if $\tau\subset\sigma$, then $(X,A)^\tau\subset (X,A)^\sigma$.
Note also that for each $\sigma$, $(X,A)^\sigma\cong X^{|\sigma|}\times A^{|n-\sigma|}$.
Next one defines the ``polyhedral product" determined by $(X,A)$ and $K$ as
$$(X,A)^K = \bigcup_{\sigma\in K}(X,A)^\sigma \subset X^n$$
If $K=\emptyset$, then $(X,A)^\emptyset=A^n$, and if
$K=\{v_0,\ldots, v_n\}$ (the vertices) and $A=*$ is a point, then $(X,*)^K$ is the wedge $\bigvee^nX$.

%\bex\label{Kcomplex}
\begin{figure}[htb]\label{Kcomplex}
\begin{center}
\begin{tikzpicture}[scale=0.40]
\fill[line width=1pt,color=black,fill=gray, fill opacity=0.4] (1,2) -- (0,-1) -- (5,0) -- cycle;
\draw [line width=1pt] (1,2)-- (0,-1);
\draw [line width=1pt] (0,-1)-- (5,0);
\draw [line width=1pt] (5,0)-- (1,2);
\draw [line width=1pt] (5,0)-- (4,4);
\draw [line width=1pt,color=black] (5,0)-- (7,4);
\draw [fill=black] (1,2) circle (2pt);
\draw[color=black] (1,2.3) node {$2$};
\draw [fill=black] (0,-1) circle (2pt);
\draw[color=black] (0,-1.3) node {$1$};
\draw [fill=black] (5,0) circle (2pt);
\draw[color=black] (5,-0.3) node {$3$};
\draw [fill=black] (4,4) circle (2pt);
\draw[color=black] (4,4.3) node {$5$};
\draw [fill=black] (7,4) circle (2pt);
\draw[color=black] (7,4.3) node {$4$};
\end{tikzpicture}
\caption{The depicted $K$ has vertex set $\{1,\ldots, 5\}$.
Simplexes of $K$ are $[1,2,3],[3,5],[3,4]$ together with their various faces. We have
$(X,A)^K  = (X,A)^{[1,2,3]}\cup (X,A)^{[3,5]}\cup X^{[3,4]}
\cong (X^3\times A^2)\cup (A^3\times X^2)\cup (A^3\times X^2)$, and (see Example \ref{kcomplexchi})
$$\chi (X,A)^K = \chi(X)^3\chi(A)^2+2\chi(X)^2\chi (A)^3 - 2\chi (X)\chi (A)^4$$
}
\end{center}
\end{figure}
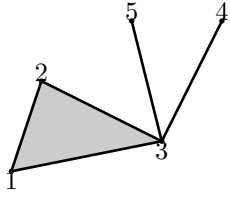

\bex\label{fat}
The fat wedge $W_d(X,n)$ is the subset of $X^n$ of all tuples with at least $d$ entries at basepoint, with $0\leq d\leq n$. Clearly $W_0(X,n)=X^n$, $W_n(X,n) = *$ and $W_1(X,n)$ is the standard one point union or wedge $\displaystyle\bigvee^nX$ in $X^n$.
This fat wedge $W_d(X,n)$ can be identified with $(X,*)^K$ where $K$ is the $n-1-d$ skeleton of the $n$-dimensional complex $\Delta_{n-1}$.
\eex

The expression for the Euler characteristic of $(X,A)^K$ for $X$ any finite CW complex appears in \cite{davis}. Below is a quick derivation of the main theorem of \cite{davis}, which was stated more restrictively
in terms of $\chi$ and for finite CW pairs $(X,A)$.

\bpr\label{chipolyhedral} Let $(X,A)$ be a pair in $\bbm$, and $K$ an abstract finite simplicial complex. Then
\begin{equation*}%\label{charpoly}
\left[ (X,A)^K\right] = \sum_{\sigma\in K}([X] - [A])^{|\sigma|}[A]^{n-|\sigma |}
\end{equation*}
the sum being over all simplices of $K$ including the empty set,
$[\emptyset]=0$. Here $|\sigma|$ is the number of vertices of the simplex $\sigma$.
When $(X,A)$ is a CW-pair, then we recover
$$\chi (X,A)^K=\sum_{\sigma\in K}(\chi (X) - \chi (A))^{|\sigma|}\chi (A)^{n-|\sigma |}$$
\epr

\begin{proof} Set
$A_\sigma := (X,A)^\sigma$. Notice that
\begin{equation}\label{property}
A_{\sigma_1}\cap A_{\sigma_2}=A_{\sigma_1\cap\sigma_2}
\end{equation}
We can set
$A^+_\sigma : = A_\sigma - \bigcup_{\tau\subset\sigma}  A_{\tau}$,
and obtain the stratification
$(X,A)^K=\bigsqcup_{\sigma\in K}A^+_\sigma$.
If $\sigma = [i_1,\ldots, i_k]$, $i_1<i_2<\cdots < i_k$, then clearly
$A_\sigma :=(X,A)^\sigma \cong X^{|\sigma|}\times A^{n-|\sigma|}$,
since this is the subset of $X^n$ consisting of tuples where
entries are in $X$ in the
$i_1,\ldots, i_k$ positions, and in $A$  in the remaining positions.
Replacing a copy of $X$ by $A$ lands immediately in
$(X,A)^\tau$ where $\tau$ is a face of $\sigma$. This means precisely that
\begin{eqnarray*}
A_\sigma^+
\cong
(X\setminus A)^{|\sigma|}A^{n-|\sigma|}
\end{eqnarray*}
so that $[A^+_\sigma] =  [X\setminus A]^k [A]^{n-k}
= ([X]- [A])^{|\sigma|}[A]^{n-|\sigma|}$. We now
sum up over all $\sigma\in K$.
\end{proof}

\iffalse\bex When $\sigma$ is the full two dimensional simplex,
$(X,*)^\sigma \cong X^3$ so that $\chi (X,*)^\sigma = \chi^3$. We can check against our formula.
\begin{eqnarray*}
\chi (X,\ast)^\sigma &=& 1 + 3\chi (X-\ast) + 3\chi (X-*)^2 + \chi (X-\ast)^3\\
&=&1 + 3(\chi -1) + 3(\chi -1)^2 + (\chi -1)^3 = \chi^3
\end{eqnarray*}
as needs to be.
\eex
\fi

\bex
By identifying $W_d(X,n)$ with $(X,*)^K$ where $K$ the $n-1-d$ skeleton of $\Delta_{n-1}$ (Example \ref{fat}), we obtain
$$\chi (W_d(X,n)) =\sum_{i=d}^{n}{{n\choose i }(\chi-1)^{n-i}}$$
In particular $\chi (\bigvee^nX)=n\chi (X) -n+1$ as is well-known.
\eex

\subsection{The Poset of face intersections}\label{mobius}

We now look at the complement of $(X,A)^K$ in $X^n$, where $K$ is an abstract simplicial complex on $n$ vertices. Obviously in this case,
$$[X^n\setminus(X,A)^K]
 = [X^n]-[(X,A)^K]=[X]^n-\sum_{\sigma\in K}([X]-[A])^{|\sigma|}[A]^{n-|\sigma |}$$
We relate this formula to a special poset associated to $K$, and give therefore an alternate derivation of Proposition \ref{chipolyhedral} or, the other way around, an easy way to compute the Mobius function of this poset.
For all definitions on poset topology, Mobius function and inversion, we refer to \cite{wachs} and also \cite{kt}.

The idea here is to think of $(X,A)^K$ as an arrangement of subspaces $(X,A)^\sigma$ in $X^{|K|}$, as $\sigma$ ranges over $K$, and then work with its poset of intersections $P_K$, ordered by reverse inclusion. We only need to look at the maximal faces, since $(X,A)^\tau\subset (X,A)^\sigma$, whenever $\tau\subset\sigma$. The poset starts with the bottom element $\hat 0$ (this corresponds to $X^{|K|}$, where all the subspaces lie), then the atoms (the elements that cover $\hat 0$) are the maximal faces of $K$. The next stage consists of elements that are the single intersections
$\sigma_1\cap\sigma_2$, etc. This is a ranked poset which at level $k$ is given by the intersection of $k$ maximal faces. The top element of the poset is $\hat 1$ and this corresponds to the intersection of all maximal facets. This intersection can be empty.

Associated to this poset $P_K$, there is a Mobius function $\mu_K (x)$, $x\in P_K$ which starts with $\mu (\hat 0) =1$,
$\mu_K (\sigma_i) = -1$, and the rest is constructed inductively according to
$$\mu_K (x) = -\sum_{y<x}\mu_K (y)$$
The following ``inclusion-exclusion'' principle is now a direct consequence of the theory of complements of arrangements (see Theorem 9.1 of \cite{kt})

\bpr\label{polyinc} Let $(X, A)$ be a pair in $\bbm$, and $K$ an abstract finite simplicial complex. Then the class of the complement is
$$[X- (X,A)^K] = \sum_{\sigma\in P_K\atop \hat 0\leq \sigma\leq \hat 1} \mu_K(\sigma)[X]^{|\sigma|} [A]^{n-|\sigma|}$$
The class $[(X,A)^K]$ is obtained by subtracting the term on the right from $[X]^{|K|}$ (this formula is the ``inclusion-exclusion formula'' for the arrangement of maximal faces).
\epr

\begin{proof}
Let $\mathcal A=\{ A_i\}_{i\in \Omega}$ be an arrangement in $Y$, where each pair $(Y,A_\alpha)$ is in $\bbm$. By Theorem 9.1 of \cite{kt}, $Y\setminus \bigcup_iA_i$ is in $\mathcal M$ and we have the general formula $\left[Y\setminus \bigcup A_i\right]
=\sum_{\hat{0}\leq \alpha\leq \hat{1}}\mu(\hat{0},\alpha)[A_\alpha]$, where $\mu$ is the Mobius function of the poset of intersections of $\mathcal A$. Here $\hat 0$ is $Y$ and $\hat 1$ is $\bigcap A_\alpha$. We apply this formula to $Y=X^{|K|}$ and to the arrangement of $A_\alpha = (X,A)^\alpha= X^{|\alpha|}\times A^{|K|-|\alpha|}$, where $\alpha$ is a maximal face of $K$.
\end{proof}

We illustrate this result in details for the complex in Fig.2.

\bex\label{kcomplexchi} Fig. 3 illustates both $P_K$ and $\mu_K$ for the complex in Fig. 2. The values of the Mobius function are indicated under each poset entry. The computation of $[X-(X,A)^K]$ is given as follows
\begin{eqnarray*}
[X-(X,A)^K]&=&[X]^5 - [(X,A)]^{[123]} - [(X,A)]^{[34]}-[(X,A)]^{[35]}+2[(X,A)]^{[3]}\\
&=&[X]^5 -[X]^3[A]^2-2[X]^2[A]^3+2[X][A]^4
\end{eqnarray*}
which gives immediately the inclusion-exclusion formula
\begin{equation}\label{formula1}
[(X,A)^K] = [X]^3[A]^2+2[X]^2[A]^3-2[X][A]^4\end{equation}
This must be checked against the formula in Proposition \ref{chipolyhedral} where $K$ is given by $$\{[123], [12], [13], [23], [34], [35], [1], [2], [3], [4], [5],\emptyset\}$$ Summing up the terms
\begin{eqnarray}\label{formula2}
\left[ (X,A)^K\right] &=& \sum_{\sigma\in K}([X] - [A])^{|\sigma|}[A]^{n-|\sigma |}\nonumber\\
&=& ([X]-[A])^3[A]^2+5([X]-[A])^2[A]^3
+5([X]-[A])[A]^4 + [A]^5
\end{eqnarray}
which reduces to \eqref{formula1}.
%\bex\label{Kcomplex}
\begin{figure}[htb]\label{mobius}
\begin{center}
\epsfig{file=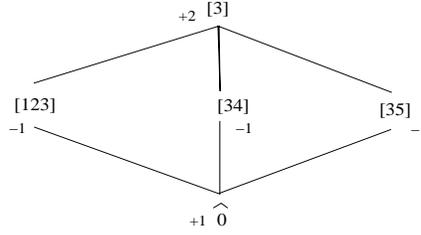,height=1.2in,width=2.2in,angle=0.0}
\caption{The poset $P_K$ and its mobius function $\mu_K$ for $K$ the complex in Fig. 2.
}
\end{center}
\end{figure}
\eex

%%%%%%%%%%%%%%%%%%%%%%%%%%%%%%

\section{Polyhedral Configuration Spaces}\label{poly2}

Unlike polyhedral products,
``Polyhedral configuration spaces'',
or ``simplicial configuration'' spaces, have only been considered in a handful of references \cite{cooper, iriye, labassi} and they form an interesting subject of research.
This is a special family of diagonal arrangements, and our terminology ``polyhedral'', instead of ``simplicial configuration spaces'' \cite{cooper}, comes from the fact that the construction of these spaces is totally analogous to that of  polyhedral products, and in some cases, they  have isomorphic lattices of intersections. Partial computations of their Euler characteristics are in \cite{labassi, iriye}. We give a streamlined approach to these computations and then generalize them in this section.% and recover an announced result in \cite{cooper2}.

Let $K$ be an abstract simplicial complex on the vertex set $\Omega:= \{1,\ldots, n\}$.
This is a collection of subsets of $\Omega=\{1,\ldots, n\}$ closed under inclusion.
Let $\sigma = \{i_1,\ldots, i_k\}\in K$,  and define
$$\Delta_\sigma (X) := \{(x_1,\ldots, x_n)\in X^n\ |\
x_{j_1}=\cdots = x_{j_{n-k}}\ \hbox{for}\ \{j_1,\ldots, j_{n-k}\} = \{i_1,\ldots, i_k\}^c\}$$
where $\{i_1,\ldots, i_k\}^c$ means the complement in $\Omega$.
Note that equating entries indexed over complements ensures that if $\sigma_1\subset\sigma_2$, then
$\Delta_{\sigma_1} (X)\subset \Delta_{\sigma_2} (X)$.
Similar to polyhedral products, we then define the
\textit{polyhedral arrangement}
\begin{equation*}
 \Delta_K(X) = \bigcup_{\sigma\in K}\Delta_\sigma (X) \ \ \ \ \ ,\ \ \ \sigma\ \hbox{simplex of $K$}
\end{equation*}
This is then a union of spaces each homeomorphic to a certain product of $X$'s.
When $X$ is a simplicial complex, by passing to a barycentric subdivision, $\Delta_K(X)$ becomes a simplicial subcomplex of $X^{|K|}$.

For convenience, we write $\Delta_\sigma (X)=\Delta_\sigma$ when no confusion arises.

\bre By construction,
\begin{equation}\label{beware}
\Delta_{\sigma_1\cap \sigma_2}\subset\Delta_{\sigma_1}\cap \Delta_{\sigma_2}
\end{equation}
but, unlike polyhedral products  \eqref{property}, this inclusion can be strict.
As an example, let $K=\Delta_3$ be the three dimensional simplex, let
$\sigma_1=\{1,2\}$ and let $\sigma_2=\{3,4\}$, then
$\Delta_{\sigma_1\cap\sigma_2}=\Delta_{\emptyset}=\{(x,x,x,x), x\in X\}$, while
$\Delta_{\sigma_1}\cap \Delta_{\sigma_2}=\{(x,x,y,y), x,y\in X\}$.
\ere

\bde The complement of the polyhedral arrangement in $X^n$ is written as in \cite{cooper}
$$M(K,X) := X^{|K|} -\Delta_K(X)$$
It can be thought of as a generalized configuration space \cite{petersen, kt}.
\ede

\bex Let $K=\Delta_{n-1}$ be the $n-1$-dimensional simplex with $n$ vertices. Then $\Delta_K(X)=X^n$. More generally,
if $K^{(d)}$ is the skeleton of dimension $d$ of $K$, then $\Delta_K(X)$ is the union in $X^{n}$ of all diagonals having one entry repeating $n-d$ times.
\eex

\bex Consider the complex $K$ in Fig. \ref{Kcomplex}. Then
\begin{eqnarray}\label{example}
\Delta_K(X) &=& \Delta_{[1,2,3]}(X)\cup \Delta_{[3,4]}(X)\cup\Delta_{[3,5]}(X)\\
&=& \{(a,b,c,x,x)\}\cup \{(y,y,a,b,y)\}
\cup\{(z,z,a,z,b)\}\nonumber
\end{eqnarray}
\eex

The space $\Delta_K(X)$ has associated to it a poset of intersections, where spaces of the posets are the various intersections $\bigcap_i \Delta_{\sigma_i}(X)$, and the partial order is given by subset inclusion. If $K=\Delta_{n-1}$ is the $n-1$-dimensional complex (with $n$-vertices), then $\Delta_K(X) = X^n$.
This space is naturally stratified by the
\begin{equation}\label{+const}
\Delta^+_\sigma = \Delta_\sigma - \bigcup_{\emptyset\neq\tau\in K}\Delta_\tau\cap\Delta_\sigma
\end{equation}

\begin{proposition}\label{mainpoly}
Suppose $K=K_1\sqcup K_2\sqcup\cdots\sqcup K_N$ is the disjoint union of $N$ non-empty complexes. We suppose that each $K_i$ is not reduced to a single simplex, and that $N\geq 3$. Then
$$[\Delta_K(X)] = \sum_i[\Delta_{K_i}(X)]- (N-1)[X]$$
\end{proposition}

\begin{proof} Take two simplices $\sigma$ and $\tau$ from $K_i$ and $K_j$ respectively, $i\neq j$. We first observe that under the hypothesis of the Proposition, their complements in $K$ must intersect. Moreover, the complement of $\tau$ contains $\sigma$, and vice-versa, so $\Delta_{\sigma}(X)\cap\Delta_{\tau}(X)\cong X$.

Consider now $\Delta_K(X)$. Any simplex $\sigma\in K$ must be in some $K_i$ for some $i$. It is convenient to write
$$\Delta_{K_i\subset K}(X)= \bigcup_{\sigma\in K_i}\Delta_\sigma(X)\subset\Delta_K(X)\subset X^{|K|}$$
It is easy to see that this space is homeomorphic to
$\Delta_{K_i}(X)\subset X^{|K_i|}$. Indeed, consider the map
$\Delta_{\sigma\in K_i}(X)\rightarrow\Delta_{\sigma\in K} (X)$, which takes the tuple $(x_1,\ldots, x_{|K_i|})$, with $x_{i_1}=x_{i_2}=\cdots = x_{i_k}=x$ if $\{i_1,\ldots, i_k\}= \sigma^c$, and extends it to $X^{|K|}$ by adding entries equal to $x$. This map is homeomorphism since $\sigma^c\in K_i$ is non-empty by the hypothesis that no $K_i$ is reduced to a single simplex.
We can then write
\begin{equation}\label{union}\Delta_K(X)=
\bigcup\Delta_{K_i\subset K}(X)\cong
\bigcup \Delta_{K_i}(X)
\end{equation}
By our first observation at the start of the proof,
$\Delta_{K_i}(X)\cap\Delta_{K_j}(X)\cong X$, and in fact any number of intersections of the $\Delta_{K_i}(X)$, for distinct choices of $K_i$, is a copy of $X$.
Using the inclusion-exclusion principle to the union
\eqref{union}, we then obtain
\begin{eqnarray*}
[\Delta_K(X)] &=&\sum_i [\Delta_{K_i}(X)] - \sum_{i\neq j}[\Delta_{K_i}(X)\cap\Delta_{K_j}(X)] + \cdots \pm \left[\bigcap_{i=1}^N\Delta_{K_i}(X)\right]\\
&=&\sum_i [\Delta_{K_i}(X)] - {N\choose 2}[X] + {N\choose 3}[X]-\cdots - (-1)^N [X]\\
&=&\sum_i [\Delta_{K_i}(X)] - (N-1)[X] - \left(\sum_{i=0}^N (-1)^i{N\choose i}\right)[X]
\end{eqnarray*}
Since the latter term is zero, the claim follows.
\end{proof}

\subsection{The case  $2(\dim K+1) < n$}
We suppose $K$ is a complex on $n$ vertices, and  $2(\dim K+1) < n$. It is direct to check that in this case $\Delta_{\sigma_1\cap\sigma_2} (X)= \Delta_{\sigma_1} (X)\cap\Delta_{\sigma_2}(X)$ for any two simplexes $\sigma_1,\sigma_2$ of $K$, and thus we can rewrite  \eqref{+const} in the nicer form
\begin{equation}\label{condition}
\Delta^+_\sigma = \Delta_\sigma - \bigcup_{\tau\subsetneq\sigma}\Delta_\tau
\end{equation}
To compute the class of this stratum, we need the following lemma.

\ble\label{wn} For $X\in\bbm$ consider
the subspace $W_n:=\{(x_1,\ldots, x_n)\in X^n\ |\ x_i\neq x_n, \forall i\neq n\}$. Then
$$[W_n] = [X]([X]-1)^{n-1}$$
\ele

\begin{proof} When $X$ is a  manifold, the projection onto the last coordinates gives a bundle over $X$ with fiber
$(X^*)^{n-1}$ where $X^*$ is $X-pt$ (the homeomorphism type of the
punctured $X$ doesn't depend on the choice of the puncture
when $X$ is a manifold, compact or not). By the multiplicative property for bundles
$[W_n] = [X]([X]-1)^{n-1}$. From this we deduce the general case. Since $W_n$ is the complement of the diagonal arrangement $\{(x_1,\ldots, x_i,\ldots, x_{n-1},x_i)\}_{1\leq i\leq n-1}$, its class $[W_n]$ is a polynomial in $[X]$ of degree $n$ and so must be $[X]([X]-1)^{n-1}$.
\end{proof}

If $\sigma\in K$, we will write $|\sigma|$ its cardinality. We use the notation $[X]^{|\emptyset|}=1$.

\bpr\label{chiiriye}
Suppose $K$ is a simplicial complex on $n$ vertices, and $2(\dim K+1) < n$, then
$$[\Delta_K(X)] = [X][(X,*)^K] = [X]\left(1+\sum_{\emptyset\neq\sigma\in K}([X]-1)^{|\sigma|}\right)$$
\epr

\begin{proof}
Consider the intersection poset for
$\Delta_K(X)=\bigcup_{\sigma\in K}\Delta_\sigma (X)$. As indicated earlier, when
$2(\dim K+1)<n$, we have that
$\Delta_{\sigma}\cap\Delta_\tau=\Delta_{\sigma\cap\tau}$ and
$\Delta^+_\sigma = \Delta_\sigma - \bigcup_{\tau\subset\sigma}\Delta_\tau$.
Passing to $K_0(\bbm )$,
$[\Delta_K(X)] = \sum_{\sigma\in K} [\Delta^+_\sigma]$. But
$\Delta^+_\sigma \cong W_{|\sigma|+1}\subset X^n$ if $\sigma\neq\emptyset$.
We can then apply lemma \ref{wn} to obtain
$$[\Delta_K(X)] = [\Delta^+_\emptyset]+\sum_{\emptyset\neq\sigma\in K} [\Delta^+_\sigma]
=[X]+ \sum_{\emptyset\neq\sigma\in K}[X]([X]-1)^{|\sigma|}
$$
which is the desired formula.

Another nice argument of proof is to notice that there is a map $\Delta_K(X)\rightarrow (X,*)^K$ which is not continuous but which  can be used to stratify $\Delta_K(X)$ by strata which are bundles with fiber $X$ over strata of $(X,*)^K$. The condition $2(\dim K+1)<n$ is precisely the condition ensuring that preimages of points are copies of $X$.
\end{proof}

As a consequence, we recover one main corollary of \cite{iriye} which was obtained through sophisticated stable splitting arguments. Our formula is more general since it is not restricted to closed manifolds.

\bco\label{chipolyarrang} Let $X$ be a manifold of dimension $m$ (i.e. no boundary) with Euler characteristic $\chi:=\chi (X)$. Let $K$ be an abstract simplicial complex on $n$ vertices, and suppose $2(\dim K+1)<n$, then if we set
$M(K,X) := X^n - \Delta_K(X)$ as in \cite{cooper}
$$\chi (M(K,X)) = \chi^n - (-1)^{m(n+1)}\chi (1+\sum_{\emptyset\neq\sigma\in K}((-1)^m\chi-1)^{|\sigma|})$$
For $X$ closed manifold of odd dimension,
$\chi (X^n\setminus\Delta_K(X))=0$, and if $X$ is closed of even dimension
$\chi (X^n - \Delta_K(X)) = \chi^n - \chi (1+\sum_{\emptyset\neq\sigma\in K}(\chi-1)^{|\sigma|})$.
\eco

\iffalse
\begin{proof}
When $X$ is a closed manifold of dimension $m$, then the formula is a direct consequence
of Corollary \ref{complementclosed}, applied to $X^n$ boundaryless of dimension $mn$, and $\Delta_K(X)$ is a closed subspace of $X^n$.
\end{proof}
\fi

\bex We illustrate our calculations on the
the line graph $K$ on $n$ vertices, seen as a $1$-dimensional simplex with $n$-vertices and $n-1$ edges. Here $\dim K=1$.
When $n=2$, and $K$ is a segment with two vertices, the $\Delta_K(X)=X^2$. For $n\geq 5$, we use  Proposition \ref{chiiriye} to obtain
%The constituent subspaces of $\Delta_K(X)$ are
%$$\Delta_{[i,i+1]} = \{(x_1,\ldots, x_n), x_{i_1}=x_{i_2}=\cdots =x_{i_{n-2}}\ \hbox{for}\ \ \ i_k\neq i,j\}$$
%and $\Delta_{[i]} = \{(x,x,\ldots, x,y,x,\ldots, x),\ y\ \hbox{in $i$-th position}\}$.
%Note that
%$\Delta_{[i]}^+\cong\conf (X,2)$ (given by all such tuples where $x\neq y$), while
%$$\Delta_{[i,i+1]}^+ =  \{(x_1,\ldots, x_n), x_i\neq x_{i_1}=x_{i_2}=\cdots =x_{i_{n-2}}\neq x_{i+1}\ \hbox{for}\ \ i_k\neq i,j\}$$
%that is $x_i,x_{i+1}$ are distinct from all the other entries which are all equal.
%This stratum can be stratified further into two strata: the set of tuples where $x_i= x_{i+1}$, and the other set where $x_i\neq x_{i+1}$. The first stratum is homeomorphic to $\conf (X,2)$, and the second to $\conf (X,3)$. To sum up,  we have the stratification
%$$\Delta_K(X)\doteqdot (n-1)\conf(X,3)\sqcup (n+n-1) )\conf(X,2)\sqcup X$$
%so that
%\begin{eqnarray*}[\Delta_K(X)]&=& (n-1)[X]([X]-1)([X]-2)+ (2n-1)[X]([X]-1)+[X]\\
%&=& [X]\left(1+n([X]-1) + (n-1)([X]-1)^2\right)
%\end{eqnarray*}
\begin{equation*}[\Delta_K(X)]=  [X]\left(1+n([X]-1) + (n-1)([X]-1)^2\right)
\end{equation*}
In fact this formula is valid for $n\geq 3$ as well since we can compute it directly, using the fact that
$\Delta_{[i]}^+\cong\conf (X,2)$,
 $1\leq i\leq n$, and $\Delta_{[i,i+1]}^+$ is stratified by one copy of $\conf (X,2)$ and one copy of $\conf (X,3)$, $1\leq i\leq n-1$.
\eex

\bco Let $X$ and $K$ as in Proposition \ref{chiiriye}, with $K$ not reduced to a single simplex, and let $\mu_K$ the poset of maximal face intersections associated to $K$.  Then
$$[M(K,X)] = \sum_{\sigma\in P_\Delta\atop\hat{0}\leq \sigma\leq \hat{1}}\mu_K(\sigma)[X]^{|\sigma|+1}$$
\eco

\begin{proof}
In the range $2(\dim K+1) < n$,  the polyhedral product describes an arrangement whose poset of intersections is $P_K$ (the maximal face poset).
By the same ``Inclusion-exclusion'' principle for arrangements (Theorem 9.1, \cite{kt}),   the class of the complement in $X^n$ is
\begin{equation}\label{grothcomparrange}
[M(K,X)] := \left[X^n\setminus \Delta_K(X)\right]
=\sum_{\tau\in P_\Delta\atop\hat{0}\leq \tau\leq \hat{1}}\mu_K(\tau)[\Delta_\tau]\\
\end{equation}
where $\mu_K$ is the Mobius function of $P_K$. This is clearly  analogous to the formula in Proposition \ref{polyinc} whereby we have replaced $[(X,A)]^\sigma$ by $[\Delta_\sigma]$. Since
$\displaystyle [\Delta_\sigma]=\begin{cases}[X]^{|\sigma|+1}&if\ K\neq \sigma \\
[X]^{|\sigma|}& if\ K=\sigma
\end{cases}$, we deduce the corollary.
\end{proof}

%%%%%%%%%%%%%%%%%%%%%%%%%%%%%%%%%%%%%%%%%%%
\section{Spaces of $0$-cycles}

In interesting recent work \cite{fww}, the authors introduce the space of $0$-cycles, which is a natural generalization of the space of rational maps (see Example \ref{rational}) and the divisor spaces in \cite{kallel2}. They show some suprising ``homological density'' results, one of which is stated as Theorem \ref{fww0} below. We give, in this final section, a  short and self-contained combinatorial derivation of an extended version of this Theorem, given in the context of the Grothendieck ring.

We first introduce the spaces in question.
Let $X$ be a connected locally compact space and fix $n,m\geq 1$. Let $\vec d$ denote a tuple of non-negative integers $(d_1,\ldots, d_m)\in\mathbb Z^m_{\geq 0}$ and let $|\vec d|:=\sum_id_i$. Let $\sp{\vec d}(X) = \prod_i\sp{d_i}(X)$. An element $\zeta$ in the symmetric product $\sp{n}(X)$ is referred to as a ``configuration of unordered points'' and is written conveniently as an abelian sum
$\zeta = \sum n_ix_i$, $\sum n_i=n$. The integer coefficient $n_i$ multiplying $x_i$ means that the point $x_i$ is listed $n_i$-times in $\zeta$ (or more if $x_j=x_i$ for some other $j$). An element of the product
$\sp{\vec d}(X)$ is therefore a ``multi-configuration'' of points.

Set $\vec d = (d_1,d_2,\ldots, d_m)$ and consider $\mathcal Z_n^{\vec d}(X)\subset\sp{\vec d}(X)$ the subset consisting of all multi-configurations $D$ of $|\vec d|$ (not necessarily distinct) points in $X$ such that:
\begin{itemize}
    \item precisely $d_i$ of the points in $D$ are labeled with the ``color'' $i$, and
    \item no point of $X$ is labeled with at least $n$ labels of every color.
\end{itemize}

In other words, $\mathcal Z_n^{\vec d}$ is the complement in $\sp{\vec d}(X)$ of the subspace
$$\{(nx + \zeta_1, nx+\zeta_2,\ldots, nx+\zeta_m)\ |\
\hbox{for some}\ x\in X, \zeta_i\in \sp{d_i-n}(X)\}$$
This is the subspace where at least one entry in each configuration has multiplicity $n$ (i.e. each color at that entry repeats $n$-times).
The following result is the main object of interest in this section.

\begin{theorem}\label{fww0}(\cite{fww}, Theorem 1.9) For $X$ a connected, oriented, smooth, even-dimensional manifold with $\dim H^*(X;\mathbb Q)<\infty$,
\begin{equation}\label{formula}
{\sum_{\vec d\in\bbz^m_{\geq 0}}\chi (\mathcal Z_n^{\vec d})x^{|\vec d|}\over
\sum_{\vec d\in\bbz^m_{\geq 0}}
\chi (\sp{\vec d}(X))x^{|\vec d|}} =  (1-x^{mn})^{\chi (X)}
\end{equation}
\end{theorem}

The above  computation took a few pages of homological calculations to derive, using tools like the Leray spectral sequence and shellability of simplicial complexes.
We will give a much shorter combinatorial derivation of the same formula valid more generally for any $[X]$ in the Grothendieck ring. Our analogous formula specializes immediately to \eqref{formula} in the case when $X$ is an even dimensional oriented manifold, since in that case $\langle X\rangle = \chi (X)$.

More precisely we prove

\begin{theorem}\label{fww} For $X\in\bbm$,
$${\sum_{\vec d\in\bbz^m_{\geq 0}}\left[\mathcal Z_n^{\vec d}\right]x^{|\vec d|}\over
\sum_{\vec d\in\bbz^m_{\geq 0}}
\left[\sp{\vec d}(X)\right]x^{|\vec d|}} =  (1-x^{mn})^{[X]}$$
Equivalently
$\displaystyle\sum_{\vec d\in\bbz^m_{\geq 0}}\left[\mathcal Z_n^{\vec d}\right]x^{|\vec d|}
=
(1-x^{mn})^{[X]}(1-x)^{-m[X]}$.
\end{theorem}

The second formula in the theorem is a direct consequence of the first, knowing that by a result of MacDonald,
$\displaystyle\sum [\sp{d}(X)]x^d = (1-x)^{-[X]}$ (\cite{gz,kt}), and so
$\sum_{\vec d\in\bbz^m_{\geq 0}}
\left[\sp{\vec d}(X)\right]x^{|\vec d|}$ is the product $m$-times of this series.

\begin{proof}
Clearly
$\mathcal Z_n^{\vec d} = \sp{\vec d}(X)$ if one of the $d_i$'s is strictly less than $n$.
In general, let's express $\mathcal Z_n^{\vec d}$ as the complement in
$\sp{\vec d}(X)$ of a ``singular set''
$$\sing{}^{\vec d} (X) = \{(nx + \zeta_1,
\ldots, nx + \zeta_m)\in\sp{\vec d}(X)\ |\ \zeta_i\in\sp{d_i-n}(X)\}$$
This singular set is stratified by the $\sing{k}^{\vec d}(X)$, for $0< k\leq \min_i\left\lfloor {d_i\over n}\right\rfloor$, where the $k$-stratum is given as follows:
$$\sing{k}^{\vec d} (X) = \{(nx_1 + \cdots +nx_k + \zeta_1,
\ldots, nx_1 + \cdots +nx_k + \zeta_m)\in\sp{\vec d}(X)\ |\ (\zeta_1,\ldots,\zeta_m)
\in\mathcal Z_n^{d_1-kn,\ldots, d_m-kn}\}$$
In other words,
$\sing{k}^{\vec d} (X)\subset\sing{}^{\vec d} (X)$
consists of all tuples of configurations in $\sp{\vec d}(X)$ sharing \textit{exactly} $k$ points, each point having multiplicity $n$ or greater.
It is convenient to write
$\vec d - n = (d_1-n,\ldots, d_m-n)$.
There is a homeomorphism
\begin{equation}\label{cong}
\sing{k}^{\vec d} (X) \cong
\sp{k}(X)\times\mathcal Z_n^{\vec d-kn}(X)
\end{equation}
where the first factor, $\sp{k}(X)$ on the left, is describing up to homeomorphism, the diagonal in $\sp{k}(X)^m$. By definition of the singular strata, we have
$$\mathcal Z_n^{\vec d}(X)=\sp{\vec d}(X)\setminus\bigsqcup_{0< k\leq \min_i\left\lfloor {d_i\over n}\right\rfloor} \sing{k}^{\vec d-kn}(X)\ \ \ \ ,\ \ \hbox{where}\ \vec d = (d_1,\ldots, d_m)$$
Combining with \eqref{cong}, we deduce immediately the equality
\begin{equation}\label{one}
\sum_{0\leq k\leq \min_i\left\lfloor {d_i\over n}\right\rfloor} \left[\sp{k}(X)\right]\cdot
\left[\mathcal Z_n^{\vec d-kn}(X)\right] =
\left[\sp{\vec d}(X)\right]
\end{equation}
In order to write generating series, we need the exponent, and each $\vec d = (d_1,\ldots, d_m)$ contributes an exponent of $|\vec d|=\sum d_i$.
The term $\sp{k}(X)$ contributes the power term $x^{km}$ (recall this is the diagonal in $\sp{(k,\ldots, k)}(X)$), while the other term
$\left[\mathcal Z_n^{\vec d-kn}(X)\right]$
contributes the power term $x^{|\vec d-kn|}$. We can rewrite \eqref{one} in terms of series coefficients as in
$$\sum_{0\leq k\leq \min_i\left\lfloor {d_i\over n}\right\rfloor} \left(\left[\sp{k}(X)\right]\cdot
\left[\mathcal Z_n^{\vec d-kn}(X)\right]\right) x^{kmn}x^{|\vec d-kn|} =
\left[\sp{\vec d}(X)\right]x^{|\vec d|}
$$
But we know that $[\mathcal Z_n^{\vec d-kn}(X)]=[\emptyset]=0$ if
$k>\min_i\lfloor {d_i\over n}\rfloor$, and so this can be further refined to
$$\sum_{k\geq 0} \left(\left[\sp{k}(X)\right]x^{kmn}\cdot
\left[\mathcal Z_n^{\vec d-kn}(X)\right]x^{|\vec d-kn|}\right) =
\left[\sp{\vec d}(X)\right]x^{|\vec d|}
$$
thus passing to series
\begin{eqnarray*}
\sum_{\vec d}\left[\sp{\vec d}(X)\right]x^{|\vec d|}&=&\sum_{k\geq 0}\sum_{\vec d-kn} \left(\left[\sp{k}(X)\right]x^{kmn}\cdot
\left[\mathcal Z_n^{\vec d-kn}(X)\right] x^{|\vec d-kn|} \right)\\
&=&\left(\sum_{k\geq 0}\left[\sp{k}(X)\right]x^{kmn}\right)\cdot\left(\sum_{\vec d}\left[\mathcal Z_n^{\vec d}(X)\right] x^{|\vec d|} \right)
\end{eqnarray*}
But then again we need recall that
$\displaystyle [\sp{k}(X)] = {k+[X]-1\choose k}$, so that
\begin{eqnarray*}
\sum_{k\geq 0}\left[\sp{k}(X)\right]x^{kmn}
&=& \sum_{k\geq 0} {k+[X]-1\choose k}x^{kmn}\\
&=&
1 + [X]x^{mn} + {[X]([X]+1)\over 2}x^{2mn} + {[X]([X]+1)([X]+2)\over 3!}x^{3mn} + \cdots\\
&=&1 -\alpha x^{mn} + {\alpha (\alpha -1)\over 2}x^{2mn} - {\alpha (\alpha - 1)(\alpha - 2)\over 3!}x^{3mn} + \cdots\ \ \hbox{where $\alpha = -[X]$}\\
&=&(1-x^{mn})^{-[X]}
\end{eqnarray*}
Putting it all together yields the result.
\end{proof}

As an immediate corollary, we recover Theorem \ref{fww0}, since for even dimensional manifolds $\langle X\rangle = \chi_c(X) = \chi (X)$.

We illustrate Theorem \ref{fww} with two key examples.

\bex This is when $m=1$ and $\mathcal Z_n^d$ is the subspace of $\sp{d}(X)$ consisting of configurations of points whose multiplicity cannot exceed $n-1$. It is shown in \cite{kt} that
$$[\mathcal Z_n^d(X)] = \sum_{\sum i\alpha_i=n, i<d}
{(\Sigma\alpha_i)!\over \alpha_{1}!\cdots \alpha_{n}!}{[X]\choose\Sigma\alpha_i}$$
from which it is deduced that
$$\displaystyle 1+\sum_{n\geq 1} [\mathcal Z_{n}^d(X)]t^d=(1-t^{n})^{[X]}(1-t)^{-[X]}$$
in full agreement with Theorem \ref{fww}.
\eex

\bex\label{rational} Consider the space of based rational maps from $\bbp^1=\bbc\cup\{\infty\}$ to itself. Any such map, of positive degree $k$, can be written as
$$f(z) = {(z-a_1)\ldots (z-a_k)\over (z-b_1)\ldots (z-b_k)}$$
and so is completely determined by the zeros $\{a_i\}$ and the poles $\{b_j\}$, with the restriction that no pole can coincide with a zero. That is, if $\rat{k}(\bbp^1)$ denotes the set of all such maps, then $\rat{k}(\bbp^1) = \mathcal Z_1^{k,k}(\bbc)$, where here $X=\bbc$ and $m=2$. Observe that $S^1$ acts freely on $\mathcal Z_1^{d_1,d_2}$ as long as $d_1d_2\neq 0$. This action is given by complex multiplication of all roots and poles by $e^{i\theta}$. Since the action is free, with quotient $Q$, we can write
$[\mathcal Z_1^{d_1,d_2}]=[S^1][Q] = 0$ if $d_1d_2\neq 0$. This gives
\begin{eqnarray*}
\sum_{\vec d\in\bbz^2_{\geq 0}}
\left[\mathcal Z_1^{\vec d}\right]x^{|\vec d|}
&=&
{\mathcal Z}_1^{0,0} + \left([\mathcal Z_1^{1,0}(\bbc)] + [\mathcal Z_1^{0,1}(\bbc)]\right)x
+ \left([\mathcal Z_1^{2,0}(\bbc)] + [\mathcal Z_1^{0,2}(\bbc)]\right)x^2  \cdots \\
&=& 1+2x + 2x^2+\cdots = {1+x\over 1-x}
\end{eqnarray*}

On the other hand, observe that
$\sp{\vec d}(\bbc ) = \sp{d_1}(\bbc)\times\sp{d_2}(\bbc)$ has the class of a point, implying that
$$
\sum_{\vec d\in\bbz^2_{\geq 0}}
\left[\sp{\vec d}(\bbc)\right]x^{|\vec d|}
=
\sum {\mathcal Z}_1^{0,0} + \left([\mathcal Z_1^{1,0}(\bbc)] + [\mathcal Z_1^{0,1}(\bbc)]\right)x
+\cdots
= 1+2x + 3x^2+\cdots = {1\over (1-x)^2}
$$
(we could also have obtained this last equation from MacDonald's formula
$\sum_{\vec d\in\bbz^2_{\geq 0}}
\left[\sp{\vec d}(\bbc)\right]x^{|\vec d|}=(1-x)^{-m[X]}$,
and plugging $[X]=[\bbc]=1$). In any case, after taking ratios
$$
{\sum_{\vec d\in\bbz^m_{>0}}\left[\mathcal Z_n^{\vec d}\right]x^{|\vec d|}\over
\sum_{\vec d\in\bbz^m_{>0}}
\left[\sp{\vec d}(X)\right]x^{|\vec d|}}
=
{{1+x\over 1-x}\over {1\over (1-x)^2}}
= (1+x)(1-x)= 1-x^2 = (1-x^2)^{[\bbc]}$$
we obtain the desired formula.
\eex
%https://arxiv.org/pdf/2012.10777.pdf

%***********************************************************
\addcontentsline{toc}{section}{Bibliography}
\bibliography{biblio}

\end{document}